\documentclass{amsart}

\usepackage[utf8]{inputenc}
\usepackage{graphicx}
\usepackage{enumerate}
\usepackage{amssymb}
\usepackage{cite}
\usepackage[dvipsnames]{xcolor}
\usepackage{color}
\usepackage{calligra}
\usepackage[T1]{fontenc}
\usepackage{latexsym}
\usepackage[all]{xy}
\usepackage{hyperref}
\usepackage{mathrsfs}
\usepackage{amsmath,amsfonts,amssymb,amsxtra}
\numberwithin{equation}{section}
\newtheorem{theorem}{\bf Theorem}

\newtheorem{remark}{\bf Remark}[section]

\usepackage{algorithmic}
\theoremstyle{definition}

\theoremstyle{remark}

\numberwithin{equation}{section}

\newcommand{\Hi}{\mathbb{H}}
\newcommand{\Sf}{\mathbb{S}}
\newcommand{\R}{\mathbb{R}}

\newcommand{\Q}{\mathbb{Q}^{2}}

\begin{document}

\title{Isoparametric hypersurfaces in product spaces}

\author{João Batista Marques dos Santos}
\address{João Batista Marques dos Santos - Departamento de Matem\'atica, Universidade de Bras\'ilia, 70910-900, Bras\'ilia-DF, Brazil}
 \email{j.b.m.santos@mat.unb.br}
\thanks{The first author was supported by Capes and CNPq. }

\author{Jo\~ao Paulo dos Santos}
\address{Jo\~ao Paulo dos Santos - Departamento de Matem\'atica, Universidade de
Bras\'ilia, 70910-900, Bras\'ilia-DF, Brazil}
\email{joaopsantos@unb.br}
\thanks{The second author was supported by CNPq grant number 315614/2021-8.}



\begin{abstract}
In this paper, we characterize and classify the isoparametric hypersurfaces with constant principal curvatures in the product spaces $ \Q_{c_{1}} \times \Q_{c_{2}}$, where $\Q_{c_{i}}$ is a space form with constant sectional curvature $c_{i}$, for $c_1 \neq c_2$. 
\end{abstract}

\maketitle

\section{Introduction}

A hypersurface $M^n$ of a Riemannian manifold $\widetilde{M}^{n+1}$ is said to be isoparametric if it has constant mean curvature as well as its nearby equidistant hypersurfaces (i.e., the correspondent mean curvatures depend only on the distance to $M$). Equivalently, we say that $M$ is isoparametric if it is the level set of some isoparametric function defined on $\widetilde{M}$. Following M. Domínguez-Vázquez \cite{notas-miguel}, the first notion of isoparametric function appeared in 1919 in the work of the Italian mathematician C. Somigliana \cite{somigliana1918sulle}, which deals with of the relations between the Huygens principle and geometric optics. This study represented the beginning of an important research line in Differential Geometry, namely the isoparametric hypersurfaces studied by renowned mathematicians such as Beniamino Segre, Élie Cartan, and Tullio Levi-Civita.

When the ambient space is a space form, i.e., a simply connected complete Riemannian manifold with constant sectional curvature, the previous definition of isoparametric hypersurface is equivalent to saying that the hypersurface has constant principal curvatures (see \cite{isoCartan} and \cite{notas-miguel}). However, in arbitrary ambient spaces of nonconstant curvature, the equivalence between isoparametric hypersurfaces and hypersurfaces with constant principal curvatures may no longer be true. For instance, Q. M. Wang, in \cite{WangExIso}, found examples of isoparametric hypersurfaces in complex projective spaces that do not have constant principal curvatures. For more examples, we refer \cite{diaz2010inhomogeneous}, \cite{diaz2013isoparametric} and \cite{ge2015filtration}. Recently, A. Rodríguez-Vázquez, in \cite{RVazquezExIso}, found an example of a non-isoparametric hypersurface with constant principal curvatures. Another example was given by the authors, in a joint work with F. Guimarães \cite{isoparametric-mcf}.

In this work, we consider the Riemannian products of 2-dimensional space forms $ \Q_{c_{1}} \times \Q_{c_{2}}$, with constant sectional curvatures $c_{1}$ and $c_2$, respectively, with $c_1 \neq c_2$, where $c_i=1,\,0$ or $-1$, $i=1,\,2$. The particular case where $c_{1} = 1$ and $c_{2} = 0$, that is, when the ambient space is $\Sf^2 \times \R^2$, was considered by J. Julio-Batalla in \cite{s2r2Batalla} where he obtained a complete classification of isoparametric hypersurfaces with constant principal curvatures. Using some ideas developed by F. Urbano in \cite{s2s2Urbano}, where isoparametric hypersurfaces of $\mathbb{S}^2 \times \mathbb{S}^2$ were classified, J. Julio-Batalla showed that if $\Sigma$ is an isoparametric hypersurface in $\Sf^2 \times \R^2$, with constant principal curvatures and unit normal $N = N_{1} + N_{2}$, then $\lvert N_{1} \rvert$ and $\lvert N_{2} \rvert$ are constant. The classification continues by showing that $\lvert N_{1} \rvert = 1$ and $\lvert N_{2} \rvert = 0$ or $\lvert N_{1} \rvert = 0$ and $\lvert N_{2} \rvert = 1$. Thus, the hypersurface families obtained are $\Sf^2 \times \R$, $\Sf^2 \times \mathbb{S}^1(r)$ (for $r \in \R^{+}$), or $\mathbb{S}^1(t) \times \R^2$ (for $t \in (0,1]$).

In this paper, we extend and improve the results of \cite{s2r2Batalla} in the following sense. Considering the ambient space $\Q_{c_{1}} \times \Q_{c_{2}}$ with $c_{1} \neq c_{2}$, we prove

\begin{theorem}\label{theo1}
Let $\Sigma$ be an isoparametric hypersurface in $\Q_{c_{1}} \times \Q_{c_{2}}$, $c_{1} \neq c_{2}$, and  unit normal $N = N_{1} + N_{2}$. Then the principal curvatures of $\Sigma$ are constant if and only if $\lvert N_{1} \rvert$ and $\lvert N_{2} \rvert$ are constant.
\end{theorem}

In addition to the converse of a result obtained by J. Julio-Batalla, which states that if $\lvert N_{1} \rvert$ and $\lvert N_{2} \rvert$ are constant, then $\Sigma$ has constant principal curvatures, Theorem \ref{theo1} also provides the equivalence for the entire class of ambient spaces $\Q_{c_{1}} \times \Q_{c_{2}}$, with $c_1 \neq c_2$. To get this Theorem, we use the theory of Jacobi fields, based on the ideas developed by M. Domínguez-Vázquez and J. M. Manzano in \cite{dominguez-manzano}, to analyze the extrinsic geometry of hypersurfaces parallel to $\Sigma$. It is interesting to note that Jacobi field theory allows us to obtain an alternative proof of J. Julio-Batalla's result. Moreover, we obtain the following general classification of isoparametric hypersurfaces with constant principal curvatures in $\Q_{c_{1}} \times \Q_{c_{2}}$, $c_{1} \neq c_{2}$, which includes the classification for $\mathbb{S}^2 \times \mathbb{R}^2$ given in \cite{s2r2Batalla}:

\begin{theorem}\label{theo2}
Let $\Sigma$ be an isoparametric hypersurface in $\Q_{c_{1}} \times \Q_{c_{2}}$, $c_{1} \neq c_{2}$, with constant principal curvatures. Then, up to rigid motions, $\Sigma$ is an open subset of one of the following hypersurfaces:
\begin{enumerate}[a)]
 \item $\mathcal{C}^{1}(\kappa_{j})\times \Q_{c_{2}}$ or $\Q_{c_{1}} \times \mathcal{C}^{1}(\kappa_{j})$, where $\mathcal{C}^{1}(\kappa_{j})$ is a complete curve with constant geodesic curvature $\kappa_j$ in $\Q_{c_{j}}$.
\item $\Psi(\mathbb{R}^3) \subset \mathbb{H}^2 \times \mathbb{R}^2$, where $\Psi : \mathbb{R}^3 \rightarrow \mathbb{H}^2 \times \mathbb{R}^2$ is an immersion given by
\begin{equation}
\begin{split}
\Psi(t,u,v) &= e^{-b\, t}(\alpha(u),\vec{0})+ \Big(\cosh(-b\, t), 0, \sinh(-b\, t), V_0t \Big) \\
& \quad + \Big(\vec{0}, p_0 + W_0 v \Big),
\end{split}\label{eq:parametrization-Psi}
\end{equation}
where $\mathbb{H}^2 \subset \mathbb{L}^3$ is given as the standard model of the hyperbolic space in the Lorentz 3-space $\mathbb{L}^3$, the curve $\alpha$ is given by $\alpha(u)=\left( \dfrac{u^2}{2},\,u,\,-\dfrac{u^2}{2} \right)$, $V_0$ and $W_0$ are constant orthogonal vectors in $\mathbb{R}^2$ such that $||W_0||=1$ and $b=\sqrt{1-||V_0||^2}$.
\end{enumerate}
\end{theorem}

Recall that, besides of the geodesics, the complete curves $\mathcal{C}^1(\kappa_j) \subset \Q_{c_{j}}$ with constant geodesic curvature are given by: $\mathbb{S}^1(t) \subset \mathbb{S}^2$ for $t \in (0,1)$; circles, horocycles or hypercycles in $\mathbb{H}^2$; and $\mathbb{S}^1(r) \subset \mathbb{R}^2$ for $r \in \R^{+}$. Regarding the hypersurfaces given in Theorem \ref{theo2}.b), geometrically, $\Psi(\mathbb{R}^3$) provides a hypersurface given as a family of geodesically parallel surfaces given by the products $\mathcal{C}^1(1) \times \mathbb{R}$, where $\mathcal{C}^1(1) \subset \mathbb{H}^2$ is a horocycle (see Remark \ref{rmk:geometry-psi}). \\

The paper is organized as follows. In Section \ref{sec2}, we provide some preliminary concepts and notations that will be used throughout the work. Section \ref{sec3} is devoted to the proof of Theorems~\ref{theo1} and \ref{theo2}. Using Jacobi field theory, we start by proving Theorem~\ref{theo1}, which characterizes isoparametric hypersurfaces with constant principal curvatures in $ \Q_{c_{1}} \times \Q_{c_{2}}$. Then, using Theorem~\ref{theo1}, we classify these hypersurfaces by proving Theorem~\ref{theo2}.

\section{Preliminary notions and results}\label{sec2}

Before proving our main results, let us present some background content on complex and product structures, the Jacobi field theory and isoparametric functions.

Let $\Q_{c_{1}}$ and $\Q_{c_{2}}$ be two 2-dimensional space forms with distinct constant sectional curvatures $c_{1}$ and $c_{2}$, respectively. For $i=1,2$, we denote by $L_{i}$ the standard complex structure in $\Q_{c_{i}}.$ If $\Q_{c_{i}}$ is the 2-dimensional sphere $\Sf^2$ of curvature $c_{i} = 1$, $L_{i}$ is given by
\begin{align*}
\begin{split}
L_{i} : \ &T\Sf^2 \longrightarrow T\Sf^2 \\
& v \longrightarrow L_{i}(v) = p \times v, 
\end{split}
\end{align*}
for  $p \in \Sf^2$, $ v \in T_{p}\Sf^2$, see \cite{dillen2012constant}. When $\Q_{c_{i}}$ is the hyperbolic space $\Hi^2$ of curvature $c_{i} = -1$, we will consider its standard Lorentzian model, i.e.,
$$\Hi^2 = \{(x_{1},x_{2},x_{3}) \in \mathbb{L}^3 \mid -x_{1}^2 + x_{2}^2 + x_{3}^2 = -1 \, \text{and} \, x_{1} > 0\},$$
where $\mathbb{L}^3$ is the 3-dimensional Minkowski space endowed with the Lorentzian cross product $\boxtimes$, defined by
$$ (a_{1},a_{2},a_{3})\boxtimes(b_{1},b_{2},b_{3})  = (a_{3}b_{2} - a_{2}b_{3},a_{3}b_{1} - a_{1}b_{3},a_{1}b_{2} - a_{2}b_{1}).$$
In this model, $L_{i}$ is given by
\begin{align*}
\begin{split}
L_{i} : \ &T\Hi^2 \longrightarrow T\Hi^2 \\
& v \longrightarrow L_{i}(v) = p \boxtimes v, 
\end{split}
\end{align*}
for  $p \in \Hi^2$, $ v \in T_{p}\Hi^2$, see \cite{dillen2012constant} and \cite{gao2021lagrangian}. Finally, if $\Q_{c_{i}}$ is the space form $\R^2$ of curvature $c_i=0$, $L_{i}$ is defined by
\begin{align*}
\begin{split}
L_{i} : \ &\R^2 \longrightarrow \R^2 \\
& v \longrightarrow L_{i}(q_{1},q_{2}) = (-q_{2},q_{1}), 
\end{split}
\end{align*}
see \cite{s2r2Batalla}.

We endow $\Q_{c_{1}} \times \Q_{c_{2}}$ with the standard product metric, denoted by $\langle , \rangle$. Moreover, given $Y \in T(\Q_{c_{1}}~\times~\Q_{c_{2}})$, we write $Y=Y^{\Q_{c_{1}}} + Y^{\Q_{c_{2}}}$, where the components $Y^{\Q_{c_{1}}}$ and $Y^{\Q_{c_{2}}}$ of $Y$ are given as its tangent parts to $\Q_{c_{1}}$ and $\Q_{c_{2}}$, respectively. We define on $\Q_{c_{1}} \times \Q_{c_{2}}$ the complex strutures
\begin{equation*}
J_{1} = L_{1} + L_{2}, \quad  J_{2} = L_{1} - L_{2},
\end{equation*}
and we denote by $\widetilde{\nabla}$ and $\widetilde{R}$ its Levi-Civita connection and curvature tensor, respectively.

Now, let us consider the product structure $P$ in $\Q_{c_{1}}~\times~\Q_{c_{2}}$ defined by $$P\Big(Y^{\Q_{c_{1}}} + Y^{\Q_{c_{2}}}\Big) = Y^{\Q_{c_{1}}} - Y^{\Q_{c_{2}}},$$
for any vector $Y \in T(\Q_{c_{1}}~\times~\Q_{c_{2}})$. Note that $P$ satisfies
\begin{equation*}
P = -J_{1}J_{2} = -J_{2}J_{1}.
\end{equation*}
Moreover, $P$ has the following properties:
\begin{align*}
&P^{2} = I \ (P \neq I), \quad \langle PY,Z \rangle = \langle Y,PZ \rangle, \quad \mbox{and} \quad (\widetilde{\nabla}_{Y}P)(Z) = 0,
\end{align*}
for any vector field $Y, Z \in T(\Q_{c_{1}}~\times~\Q_{c_{2}})$. Using the product structure $P$, $ \widetilde{R} $ is given  by
\begin{align*}
\widetilde{R}(V,W,Z,Y) &= \frac{c_{1}}{4}\biggr\{\langle V,PY + Y \rangle\langle PW + W,Z \rangle - \langle W,PY + Y \rangle\langle PV + V,Z \rangle\biggr\} \nonumber \\
& \quad + \frac{c_{2}}{4}\biggr\{\langle V,PY - Y \rangle\langle PW - W,Z \rangle - \langle W,PY - Y \rangle\langle PV - V,Z \rangle\biggr\},
\end{align*}
where $V, W, Z, Y \in T(\Q_{c_{1}} \times \Q_{c_{2}})$, see \cite{dillen2012constant}.

Let $\Sigma^3 \subset \Q_{c_{1}} \times \Q_{c_{2}}$ be an oriented hypersurface with unit normal vector $N = N_{1} + N_{2}$ and Levi-Civita connection $ \nabla $. We define in $\Sigma^3$ a smooth function $C$ and a tangent vector field $X$ by
\begin{equation}\label{funC}
C = \langle PN,N \rangle \quad \mbox{and} \quad X = PN - CN.
\end{equation}

Observe that $X$ is the tangential component of $PN$ and $\lvert X \rvert^2 = 1 - C^2$, which implies $-1 \leq C \leq 1 $. 

Using the curvature tensor of $\Q_{c_{1}} \times \Q_{c_{2}}$ and the vector field $X$ defined above, the Codazzi equation of $\Sigma$ is given by
\begin{align*}
	\nabla S(V,W,Z) - \nabla S(W,V,Z) &= \widetilde{R}(V,W,Z,N),
\end{align*}
where
\begin{align*}
\widetilde{R}(V,W,Z,N)	&= \frac{c_{1}}{4}\biggr\{\langle V, PN + N \rangle \langle PW + W , Z \rangle - \langle W, PN + N \rangle \langle PV + V , Z \rangle \biggr\}\\
	& \quad + \frac{c_{2}}{4}\biggr\{\langle V, PN - N \rangle \langle PW - W , Z \rangle - \langle W, PN - N \rangle \langle PV - V , Z \rangle \biggr\}\\
	&= \frac{c_{1}}{4}\biggr\{\langle V, X \rangle \langle PW + W , Z \rangle - \langle W, X \rangle \langle PV + V , Z \rangle \biggr\}\\
	& \quad + \frac{c_{2}}{4}\biggr\{\langle V, X \rangle \langle PW - W , Z \rangle - \langle W, X \rangle \langle PV - V , Z \rangle\biggr\},
\end{align*}
with $V, W, Z \in T\Sigma$.

In this work, we will use the Jacobi field theory to analyze the extrinsic geometry of hypersurfaces equidistant to the hypersurface $\Sigma$. In what follows, we will give a brief description of this theory. For more details, we refer to \cite{bookOlmosCia, notas-miguel}.

Given a hypersurface $\Sigma^n$ of a Riemannian manifold $\widetilde{M}^{n+1}$ with unit normal vector field $N$, let $\varepsilon$ be a positive real number and, for $r \in (-\varepsilon, \varepsilon)$, consider the application
\begin{equation}\label{paralellhypersurfaces}
\begin{array}{rcl}
     \Phi_r: \Sigma^n & \rightarrow & \widetilde{M}^{n+1}, \\
             p &\mapsto& \exp_{p}(rN_{p}), 
\end{array}
\end{equation}
where $\exp_{p}: T_{p}\widetilde{M} \rightarrow \widetilde{M}$ denotes the exponential map of $\widetilde{M}^{n+1}$ at $p \in \Sigma$. For $\varepsilon>0$ small enough, the map $\Phi_r$ is smooth and it parametrizes the parallel displacement of $\Sigma$ at an oriented distance $r$ in the direction $N$. The parallel hypersurface $\Phi_r(\Sigma)$ will be denoted by $\Sigma_r$.

Let $\gamma_p: I \rightarrow \widetilde{M}$ be the geodesic parametrized by arc length with $0 \in I \subset \R$, $\gamma_p(0)=p \in \Sigma$ and $\dot{\gamma_p}(0) = N_{p}$. Let $\zeta_Y$ be the Jacobi field along $\gamma_p$ with initial conditions given by
$$
\zeta_Y(0)=Y, \, \textnormal{ and } \, \zeta_Y'(0) = - AY,
$$
where $A$ is the shape operator of $\Sigma$ associated with $N$. Then, a unit normal vector to $\Sigma_r$ at $\gamma_p(r)$ is given by $\dot{\gamma_p}(r)$ and its correspondent shape operator satisfies 
$$
A_r \zeta_Y(r) = - \zeta_Y'(r).
$$
If we write $\zeta_Y(r)=D(r)\widetilde{P}_Y(r)$, where $D(r)$ is an endomorphism acting on $T_{\gamma_p(r)}\Sigma_r$ and $\widetilde{P}_Y(r)$ is the parallel transport of $Y$ along $\gamma_p$, then we have
\begin{equation}\label{A-paralell}
A_r = -(D'\circ D^{-1})(r).
\end{equation}
Consequently, by the Jacobi formula, the mean curvature of the hypersurface $\Sigma_r$ is given by
\begin{equation}
h(r) = - \dfrac{(\det D)'}{n \det D}(r). \label{H-parallel}
\end{equation}

Finally, we introduce the notion of isoparametric function. A non-constant smooth function $f : \widetilde{M}^{n+1} \longrightarrow \mathbb{R} $ is called isoparametric if the gradient and the Laplacian of $f$ satisfy
$$\lvert \nabla f \lvert^{2} = a(f) \quad \mbox{and} \quad \Delta f = b(f),$$
where $a,b: I \subset \mathbb{R} \longrightarrow \mathbb{R}$ are smooth functions. The smooth hypersurfaces $\Sigma_{r} = f^{-1}(r)$ for $r$ regular value of $f$ are called isoparametric hypersurfaces. In this case, the unit normal vector field is given by $N = \frac{\nabla f}{\lvert \nabla f \rvert}$. We observe that, by the conditions under the gradient and the Laplacian given in the definition of an isoparametric function, $\Sigma_r$ has constant mean curvature for each $r$ (i.e., depending only on $r$) and $N$ is a geodesic field, see \cite{notas-miguel}.

\section{Proof of the main results}\label{sec3}

To prove Theorems \ref{theo1} and \ref{theo2}, we combine the techniques developed by F. Urbano \cite{s2s2Urbano}, J. Julio-Batalla \cite{s2r2Batalla}, and Domínguez-Vázquez and Manzano \cite{dominguez-manzano}.

\begin{proof}[Proof of Theorem~\ref{theo1}] Let $\Sigma$ be an isoparametric hypersurface in $\Q_{c_{1}} \times \Q_{c_{2}}$ with $c_{1}\neq c_{2}$ and unit normal $N = N_{1} + N_{2}$. In order to prove Theorem~\ref{theo1}, it is enough to show that the principal curvatures of $\Sigma$ are constant if and only if the function $C$, given in \eqref{funC}, is constant. In fact, as $\lvert N_{1} \rvert^2 = \frac{1+C}{2}$ and $\lvert N_{2} \rvert^2 = \frac{1-C}{2}$, it follows that $\lvert N_{1} \rvert$ and $\lvert N_{2} \rvert$ are constant if and only if $C$ is constant.

Recall that the family of hypersurfaces parallel to $\Sigma$ in the direction of $N$ is given by \eqref{paralellhypersurfaces} and the parallel hypersurface at an oriented distance $r$ is denoted by $\Sigma_r$. We first observe that, since $\Sigma$ is isoparametric and the product structure $P$ is parallel, the function $C$, defined on the family of parallel hypersurfaces, does not depend on the displacement parameter $r$, once $N(C) = 0$. In fact, since $C = \langle PN, \, N \rangle$ and $\nabla_{N}N=0$, we have
$$N(C)= \langle \nabla_{N}N, PN \rangle + \langle N, P\nabla_{N}N \rangle = 0.$$

Now we prove that $C$ is constant along $\Sigma$. Let us recall that $|C| \leq 1$. Consider the open set
$$ U = \big\{ p \in \Sigma  \mid C^2(p) < 1 \big\}. $$
We can assume that $U \neq \varnothing$, otherwise $C^2=1$ on $\Sigma$. In this case, let us take in $U$ the following orthonormal frame
\begin{equation*}
B =  \biggr\{ B_{1} = \frac{X}{\sqrt{1 - C^2}}, B_{2} = \frac{J_{1}N + J_{2}N}{\sqrt{2(1+C)}}, B_{3} = \frac{J_{1}N - J_{2}N}{\sqrt{2(1-C)}} \biggr\},
\end{equation*}
where $X = PN - CN.$

Given $p \in \Sigma$, let $\gamma_{p}$ be a geodesic of $\Q_{c_{1}} \times \Q_{c_{2}}$ with $\gamma_p(0) = p$ and $\dot{\gamma}_p(0) = N_p$. By the definition of $\Sigma_r$ we have that $\dot{\gamma}_{q}(r)$ is a normal vector to $\Sigma_{r}$ at $\gamma_q(r)$. Thus, we can extend the unit normal $N$ to $U\times(-\epsilon,\epsilon)$ by $N_{\gamma_{q}(r)} = \dot{\gamma}_{q}(r)$, $q \in U$. Consequently, we also can extend the fields $B_{i}$.

Recall that a Jacobi field along $\gamma_{p}$ is a vector field $\xi$ satisfying the Jacobi equation $\xi'' + R(\xi,\dot{\gamma}_{p})\dot{\gamma}_{p} = 0$. For each $j \in \{1,2,3\}$, take the Jacobi field $\xi_{j}$ along $\gamma_{p}$ with the initial conditions
\begin{equation}\label{jacobi-initialcondicions}
\xi_{j}(0) = B_{j} \quad \mbox{and} \quad \xi_{j}'(0) = -AB_{j},
\end{equation}
where $A$ is the shape operator of $\Sigma$ associated with $N$.

Since these initial conditions are orthogonal to $\dot{\gamma}_{p}(0)$, each Jacobi field $\xi_{j}$ is also orthogonal to $N_{\gamma_{p}(r)} = \dot{\gamma}_{p}(r)$ and, hence, it can be written as 
$$\xi_{j} = b_{1j}B_{1} + b_{2j}B_{2} + b_{3j}B_{3},$$
for certain smooth functions $b_{ij}$ on $(-\epsilon,\epsilon)$.

Let us observe that $\nabla_{N}B_{i}=0$, for all $i = 1,2,3$. In fact, since $N(C)=0$ and $P$ is parallel, we have $\nabla_N X=0$, which implies $\nabla_N B_1=0$. Furthermore, since $J_i$ is also parallel, for $i=1,\,2$, we conclude that $\nabla_N B_j=0$, $j=2,\,3.$ Thus, we have, on the one hand,
\begin{equation}\label{jacobi-second-derivative}
\xi_{j}'' = b_{1j}''B_{1} + b_{2j}''B_{2} + b_{3j}''B_{3}.
\end{equation}
On the other hand, if we denote by $R^{c_i}$ the curvature tensor of $\Q_{c_{i}}$, we get
\begin{align*}
\widetilde{R}(B_{1},N)N &= R^{c_1}(B_{1}^{c_{1}},N_{1})N_{1} + R^{c_2}(B_{1}^{c_{2}},N_{2})N_{2} \\
&= \frac{1}{8\sqrt{1 - C^2}}\biggr(R^{c_{1}}(X + PX,N + PN)(N + PN) \\
& \quad + R^{c_{2}}(X - PX,N - PN)(N - PN) \biggr) \\
&= 0,
\end{align*}
since $X + PX = (1 - C)(N + PN)$ and  $X - PX = -(1 + C)(N - PN)$. Now, using the curvature tensor formula of a manifold of constant sectional curvature, we get
\begin{align*}
\widetilde{R}(B_{2},N)N &= \frac{c_{1}\lvert N + PN \rvert^{2}}{4}B_{2}, \\
\widetilde{R}(B_{3},N)N &= \frac{c_{2}\lvert N - PN \rvert^{2}}{4}B_{3}.
\end{align*}
Therefore,
\begin{align}\label{curvaturetensorRN}
\begin{split}
\widetilde{R}(\xi_{j},\dot{\gamma}_{p})\dot{\gamma}_{p} &= \widetilde{R}(\xi_{j},N)N \\
&= b_{1j}R(B_{1},N)N + b_{2j}R(B_{2},N)N + b_{3j}R(B_{3},N)N \\
&= b_{2j}\frac{c_{1}\lvert N + PN \rvert^{2}}{4}B_{2} + b_{3j}\frac{c_{2}\lvert N - PN \rvert^{2}}{4}B_{3} \\
&= b_{2j}\frac{c_{1}(1 + C)}{2}B_{2} + b_{3j}\frac{c_{2}(1 - C)}{2}B_{3}.
\end{split}
\end{align}
Since $\xi_{j}$ is a Jacobi field, we have from \eqref{jacobi-second-derivative} and \eqref{curvaturetensorRN} the following homogeneous linear system of ordinary differential equations
\begin{equation}\label{systemequations}
b_{1j}'' = 0, \quad b_{2j}'' + \delta_{1}b_{2j} = 0, \quad b_{3j}'' + \delta_{2}b_{3j} = 0,
\end{equation}
where $\delta_{1} = \frac{c_{1}(1 + C)}{2}$ and $\delta_{2} = \frac{c_{2}(1 - C)}{2}$. 

In the sequence, we describe the initial conditions of the system \eqref{systemequations}. Firstly, as $\xi_{j}(0) = B_{j}$, we get 
\begin{align}\label{initialconditions-bij}
\begin{array}{llll}
&b_{11}(0) = 1, &b_{12}(0) = 0, &b_{13}(0) = 0,\\
&b_{21}(0) = 0, &b_{22}(0) = 1, &b_{23}(0) = 0,\\
&b_{31}(0) = 0, &b_{32}(0) = 0, &b_{33}(0) = 1.
\end{array}
\end{align}

Secondly, let the shape operator of $\Sigma$ be determined by the relations $A B_{i} = \sigma_{i1}B_{1} + \sigma_{i2}B_{2} + \sigma_{i3}B_{3}$, for certain smooth functions $\sigma_{ij}$. Since $A$ is symmetric, we have $\sigma_{12} = \sigma_{21}$, $\sigma_{13} = \sigma_{31}$ and $\sigma_{32} = \sigma_{23}$. 
Furthermore, taking into account that $\xi_{j}' = \widetilde{\nabla}_{N}\xi_{j} = -A \xi_{j}$, we obtain
\begin{align}\label{initialconditions-b'ij}
\begin{array}{llll}
&b_{11}'(0) = -\sigma_{11}, &b_{12}'(0) = -\sigma_{21}, &b_{13}'(0) = -\sigma_{31},\\
&b_{21}'(0) = -\sigma_{12}, &b_{22}'(0) = -\sigma_{22}, &b_{23}'(0) = -\sigma_{23},\\
&b_{31}'(0) = -\sigma_{13}, &b_{32}'(0) = -\sigma_{23}, &b_{33}'(0) = -\sigma_{33}.
\end{array}
\end{align}
With the initial conditions \eqref{initialconditions-bij} and \eqref{initialconditions-b'ij}, the solution of system \eqref{systemequations} is given by
\begin{align}\label{solution-system}
\begin{split}
b_{11}(r) &= -\sigma_{11}r + 1, \\
b_{12}(r) &= -\sigma_{12}r, \\
b_{13}(r) &= -\sigma_{13}r, \\
b_{21}(r) &= -\sigma_{12}S_{\delta_{1}}(r), \\
b_{22}(r) &= -\sigma_{22}S_{\delta_{1}}(r) + C_{\delta_{1}}(r), \\
b_{23}(r) &= -\sigma_{32}S_{\delta_{1}}(r), \\
b_{31}(r) &= -\sigma_{13}S_{\delta_{2}}(r), \\
b_{32}(r) &= -\sigma_{32}S_{\delta_{2}}(r), \\
b_{33}(r) &= -\sigma_{33}S_{\delta_{2}}(r) + C_{\delta_{2}}(r),
\end{split}
\end{align}

where we consider the auxiliary functions
\begin{equation*}
S_{\delta_{i}}(r)=\begin{cases}\frac{1}{\sqrt{-\delta_{i}}}\sinh(r\sqrt{-\delta_{i}})&\text{if } \delta_{i} < 0,\\
\frac{1}{\sqrt{\delta_{i}}}\sin(r\sqrt{\delta_{i}})&\text{if } \delta_{i} > 0,
\end{cases}\qquad 
C_{\delta_{i}}(r)=\begin{cases}\cosh(r\sqrt{-\delta_{i}})&\text{if } \delta_{i} < 0,\\
\cos(r\sqrt{\delta_{i}})&\text{if } \delta_{i} > 0.
\end{cases}
\end{equation*}
for $i \in \{1,2\}$.

For every $r$, the shape operator $A_{r}$ of $\Sigma_{r}$ with respect to the normal $\gamma_{p}'(r)$ is given by \eqref{A-paralell}, where $D(r)$ is linear endomorphism of $T_{\gamma_{p}(r)}\Sigma_{r}$, determined by the relations
\begin{equation*}
D(r)B_{j}(\gamma_{p}(r)) = \xi_{j}(r), \quad D'(r)B_{j}(\gamma_{p}(r)) = \xi_{j}'(r).
\end{equation*}

Considering the orthonormal basis $\{B_{1}(\gamma_{p}(r)), B_{2}(\gamma_{p}(r)), B_{3}(\gamma_{p}(r))\}$ of $T_{\gamma_{p}(r)}\Sigma_{r}$, the matrix form of the operator $D(r)$ is given by
\begin{align}\label{matrix-operatorD}
D(r) =
\left(
\begin{array}{ccc}
b_{11}(r) & b_{12}(r) & b_{13}(r) \\
b_{21}(r) & b_{22}(r) & b_{23}(r) \\
b_{31}(r) & b_{32}(r) & b_{33}(r)
\end{array}
\right), 
\end{align}

From now on, our strategy is given as follows. Firstly, we are going to get explicitly the formulas of $\det D(r)$ and $\frac{d}{dr}(\det D(r))$ in terms of the functions $b_{ij}$ and its derivatives. Secondly, will apply such formulas to construct
\begin{equation*}
f(r) = \frac{d}{dr}(\det D(r)) + 3h(r) \det D(r),
\end{equation*}
which vanishes identically on $(-\epsilon, \epsilon)$, by equation \eqref{H-parallel}. Finally, we will use the fact that $f \equiv 0$ as well as its derivatives to obtain some algebraic relations between the components of $A$ on the basis $\left\{ B_i \right\}_{i=1}^3$ and the function $C$.

From \eqref{solution-system} and \eqref{matrix-operatorD}, we have that
\begin{align*}
\det D(r) &= A_{1}rS_{\delta_{1}}(r)S_{\delta_{2}}(r) + A_{2}rS_{\delta_{1}}(r)C_{\delta_{2}}(r) + A_{3}rS_{\delta_{2}}(r)C_{\delta_{1}}(r) \\
& \quad + A_{4}S_{\delta_{1}}(r)S_{\delta_{2}}(r) - \sigma_{11}rC_{\delta_{1}}(r)C_{\delta_{2}}(r) - \sigma_{22}S_{\delta_{1}}(r)C_{\delta_{2}}(r) \\
& \quad - \sigma_{33}S_{\delta_{2}}(r)C_{\delta_{1}}(r) + C_{\delta_{1}}(r)C_{\delta_{2}}(r),
\end{align*}
where
\begin{align}\label{eq:Ai}
\begin{array}{lll}
& A_{1} = -\det A,  & A_{2} = \sigma_{11}\sigma_{22} - \sigma^2_{12}, \\
& A_{3} = \sigma_{11}\sigma_{33} - \sigma^2_{13},  & A_{4} = \sigma_{22}\sigma_{33} - \sigma^2_{23}.
\end{array}
\end{align}

Now, taking into account that $S'_{\delta_{i}}(r) = C_{\delta_{i}}(r)$ and $C'_{\delta_{i}}(r) = -\delta_{i}S_{\delta_{i}}(r)$, we obtain
\begin{align*}
\frac{d}{dr}(\det D(r)) &= A_{1}\left(S_{\delta_{1}}(r)S_{\delta_{2}}(r) + rC_{\delta_{1}}(r)S_{\delta_{2}}(r) + rS_{\delta_{1}}(r)C_{\delta_{2}}(r)\right) \\
& \quad + A_{2}\left(S_{\delta_{1}}(r)C_{\delta_{2}}(r) + rC_{\delta_{1}}(r)C_{\delta_{2}}(r) - r\delta_{2}S_{\delta_{1}}(r)S_{\delta_{2}}(r)\right) \\
& \quad + A_{3}\left(S_{\delta_{2}}(r)C_{\delta_{1}}(r) + rC_{\delta_{2}}(r)C_{\delta_{1}}(r) - r\delta_{1}S_{\delta_{2}}(r)S_{\delta_{1}}(r)\right) \\
& \quad + A_{4}\left(C_{\delta_{1}}(r)S_{\delta_{2}}(r) + S_{\delta_{1}}(r)C_{\delta_{2}}(r)\right) \\
& \quad - \sigma_{11}\left(C_{\delta_{1}}(r)C_{\delta_{2}}(r) - r\delta_{1}S_{\delta_{1}}(r)C_{\delta_{2}}(r) - r\delta_{2}C_{\delta_{1}}(r)S_{\delta_{2}}(r)\right) \\
& \quad - \sigma_{22}\left(C_{\delta_{1}}(r)C_{\delta_{2}}(r) - \delta_{2}S_{\delta_{1}}(r)S_{\delta_{2}}(r)\right) \\
& \quad - \sigma_{33}\left(C_{\delta_{2}}(r)C_{\delta_{1}}(r) - \delta_{1}S_{\delta_{2}}(r)S_{\delta_{1}}(r)\right) \\
& \quad - \delta_{1}S_{\delta_{1}}(r)C_{\delta_{2}}(r) - \delta_{2}C_{\delta_{1}}(r)S_{\delta_{2}}(r).
\end{align*}
Thus, the function $f$ is given explicitly as
\begin{align}\label{function-f}
\begin{split}
f(r) &=
A_{1}\big(S_{\delta_{1}}(r)S_{\delta_{2}}(r) + rC_{\delta_{1}}(r)S_{\delta_{2}}(r) + rS_{\delta_{1}}(r)C_{\delta_{2}}(r) \\
& \quad + 3rh(r)S_{\delta_{1}}(r)S_{\delta_{2}}(r)\big) \\
& \quad + A_{2}\big(S_{\delta_{1}}(r)C_{\delta_{2}}(r) + rC_{\delta_{1}}(r)C_{\delta_{2}}(r) - r\delta_{2}S_{\delta_{1}}(r)S_{\delta_{2}}(r) \\
& \quad + 3rh(r)S_{\delta_{1}}(r)C_{\delta_{2}}(r)\big) \\
& \quad + A_{3}\big(S_{\delta_{2}}(r)C_{\delta_{1}}(r) + rC_{\delta_{2}}(r)C_{\delta_{1}}(r) - r\delta_{1}S_{\delta_{2}}(r)S_{\delta_{1}}(r) \\
& \quad + 3rh(r)S_{\delta_{2}}(r)C_{\delta_{1}}(r)\big) \\
& \quad + A_{4}\big(C_{\delta_{1}}(r)S_{\delta_{2}}(r) + S_{\delta_{1}}(r)C_{\delta_{2}}(r) + 3h(r)S_{\delta_{1}}(r)S_{\delta_{2}}(r)\big) \\
& \quad - \sigma_{11}\big(C_{\delta_{1}}(r)C_{\delta_{2}}(r) - r\delta_{1}S_{\delta_{1}}(r)C_{\delta_{2}}(r) - r\delta_{2}C_{\delta_{1}}(r)S_{\delta_{2}}(r) \\
& \quad + 3rh(r)C_{\delta_{1}}(r)C_{\delta_{2}}(r)\big) \\
& \quad - \sigma_{22}\big(C_{\delta_{1}}(r)C_{\delta_{2}}(r) - \delta_{2}S_{\delta_{1}}(r)S_{\delta_{2}}(r) + 3h(r)S_{\delta_{1}}(r)C_{\delta_{2}}(r)\big) \\
& \quad - \sigma_{33}\big(C_{\delta_{2}}(r)C_{\delta_{1}}(r) - \delta_{1}S_{\delta_{2}}(r)S_{\delta_{1}}(r) + 3h(r)S_{\delta_{2}}(r)C_{\delta_{1}}(r)\big) \\
& \quad - \delta_{1}S_{\delta_{1}}(r)C_{\delta_{2}}(r) - \delta_{2}C_{\delta_{1}}(r)S_{\delta_{2}}(r) + 3h(r)C_{\delta_{1}}(r)C_{\delta_{2}}(r).
\end{split}
\end{align}

As $f \equiv 0$, so is its derivative. Then, taking the derivative in \eqref{function-f} and applying at $r=0$, we obtain the following relation:
\begin{align}
0 = f'(0) &= 2(A_{2} + A_{3} + A_{4}) - 9h^{2}(0) + 3h'(0) - (\delta_{1} + \delta_{2}), \label{eq:f'(0)}
\end{align}
where $h(0)$ is the mean curvature of $\Sigma$.

Note that $A_{i}$, $\delta_{i}$, $h (0)$ and $h'(0)$, depend only, in principle, of the base point $p \in \Sigma$. However, by assumption, $\Sigma$ is isoparametric and hence, $h (0)$ and $h'(0)$ are constants throughout $\Sigma$, that is, it is independent of the chosen base point $p \in \Sigma$ of normal geodesic $\gamma_{p}$.

Furthermore, observe that
\begin{equation*}
9h^{2}(0) = \sigma^{2}_{11} + \sigma^{2}_{22} + \sigma^{2}_{33} +2(\sigma_{11}\sigma_{22} + \sigma_{11}\sigma_{33} + \sigma_{22}\sigma_{33}),
\end{equation*}
and
\begin{equation*}
tr (A^2) = \sigma^{2}_{11} + \sigma^{2}_{22} + \sigma^{2}_{33} + 2(\sigma^{2}_{12} + \sigma^{2}_{13} + \sigma^{2}_{23}).
\end{equation*}
Thus, by the definitions of the $A_i's$ in \eqref{eq:Ai}, we have $2(A_{2} + A_{3} + A_{4}) - 9h^{2}(0) = -tr (A^2)$. Substituting in \eqref{eq:f'(0)}, we get
\begin{equation*}
tr (A^2) = 3h'(0) - (\delta_{1} + \delta_{2}),
\end{equation*}
where $\delta_{1} + \delta_{2} = \frac{1}{2}(C(c_{1} - c_{2}) + c_{1} + c_{2})$.

Therefore, if $\Sigma$ has constant principal curvatures $\mu_{1}$, $\mu_{2}$, $\mu_{3}$, then $tr(A^2) = \mu_{1}^2 + \mu_{2}^2 + \mu_{3}^2$ is constant and hence, $C$ is constant, since $c_1 \neq c_2$.

Conversely, suppose $C$ is constant. Since the gradient of the function $C$ is given by $\nabla C = -2A(X)$ (see [Lemma $1$, \cite{s2s2Urbano}]), then $ A(X) = 0$. Therefore, $ \sigma_{1j} = \sigma_{j1} = 0$,  for all $j = 1,2,3$. Thus, we have $A_{1} = A_{2} = A_{3} = 0$ and we can rewrite \eqref{eq:f'(0)} as
\begin{align*}
0 = 2A_{4} - 9h^{2}(0) + 3h'(0) - (\delta_{1} + \delta_{2}), \label{rewrited-eq:f'(0)}
\end{align*}
and, as a consequence, we have that $A_{4}$ is constant.

Moreover, as $ \sigma_{1j} = \sigma_{j1} = 0$, the characteristic polynomial $Q_{A}$ of $A$ is given by
$$Q_{A}(\lambda) = -\lambda^{3} + 3h(0)\lambda^{2} - A_{4}\lambda.$$
Therefore, since $A_{4}$ is constant, it follows that the principal curvatures of $\Sigma$ are constant.
\end{proof}

\begin{proof}[Proof of Theorem~\ref{theo2}]
Let $\Sigma$ be an isoparametric hypersurface in $\Q_{c_{1}} \times \Q_{c_{2}}$ with constant principal curvatures. By Theorem~\ref{theo1}, we have that $C$ is constant. If $C=1$ we have $PN = N$, and thus, $N = (N_{1},0)$. If $C=-1$ we have $PN = -N$, and then, $N = (0, N_{2})$. In such cases, $\Sigma$ is an open subset of $\mathcal{C}^{1}(\kappa_{j})\times \Q_{c_{2}}$ or $\Q_{c_{1}} \times \mathcal{C}^{1}(\kappa_{j})$, respectively, where $\mathcal{C}^{1}(\kappa_{j})$ is a complete curve in $\Q_{c_{j}}$ of constant geodesic curvature $\kappa_{j}$. In fact, let us suppose that $N=(N_1,0)$, then $\Sigma$ is an open subset of $\mathcal{C}^{1}\times \Q_{c_{2}}$, where $\mathcal{C}^{1}$ is a regular curve in $\Q_{c_1}$. Let $\psi$ be a parametrization by arc length of $\mathcal{C}^{1}$, with unit normal vector $n_{\psi} = \pm N_1$. Let $\{e_1,e_2,e_3\}$ a orthonormal frame in $\mathcal{C}^{1}\times \Q_{c_{2}}$, with $e_1 = \psi'$ and $\{e_2,e_3\}$ an orthonormal basis in $\Q_{c_{2}}$. If we denote the shape operator of $\Sigma$ by $A$, considering without loss of generality that $N_{1}= n_{\psi}$, we have
\begin{align*}
Ae_1 &= -\widetilde{\nabla}_{e_1}N_1 = -\widetilde{\nabla}^{\Q_{c_{1}}}_{\psi'}n_{\psi} = \kappa_{j}\psi' = \kappa_{j}e_1, \\
Ae_2 &= -\widetilde{\nabla}_{e_2}N_1 = 0, \\
Ae_3 &= -\widetilde{\nabla}_{e_3}N_1 = 0.
\end{align*}
Therefore, the curvature $\kappa_j$ of $\mathcal{C}^1$ is a principal curvature of $\Sigma$, which implies that $\kappa_j$ is constant. The case where $N=(0,N_2)$ is analogous.

In the sequence, we are going to prove that, if $|C|<1$, the only remaining possibility is the case when one $c_i$ is negative. Therefore, in what follows, let us assume that $ C \in (-1,1)$. In this case, as in the proof of Theorem \ref{theo1}, let us consider the frame
\begin{equation*}
B =  \left\{ B_{1} = \frac{X}{\sqrt{1 - C^2}}, B_{2} = \frac{J_{1}N + J_{2}N}{\sqrt{2(1+C)}}, B_{3} = \frac{J_{1}N - J_{2}N}{\sqrt{2(1-C)}} \right\},
\end{equation*}
and the function $f$ given in \eqref{function-f}. Again, taking derivatives in \eqref{function-f} and applying them at $r=0$, we obtain the following relations:
\begin{align}
0 = f'(0) &= 2(A_{2} + A_{3} + A_{4}) - 9h^{2}(0) + 3h'(0) - (\delta_{1} + \delta_{2}), \label{eq:f'(0)1}\\
0 = f''(0) &= 6A_{1} + 6h(0)(A_{2} + A_{3} + A_{4}) - 18h'(0)h(0) + 2\sigma_{11}(\delta_{1} + \delta_{2}) \label{eq:f''(0)1} \\
& \quad + 2\sigma_{22}\delta_{2} + 2\sigma_{33}\delta_{1} + 3h''(0), \nonumber
\end{align}
where the functions $A_i$, $i=1,\dots,4$, are given in \eqref{eq:Ai}.

Let us recall that as $C$ is constant we have $\sigma_{1i}=\sigma_{i1}=0$, which imply that $A_{1} = A_{2} = A_{3} = 0$. Moreover, since $h(0)$ is the mean curvature of $\Sigma$, we also conclude that
\begin{equation}\label{eq:h(0)}
3h(0) = \sigma_{22} + \sigma_{33}.
\end{equation}
Thus, we can rewrite \eqref{eq:f'(0)1} and \eqref{eq:f''(0)1} as follows:
\begin{align}
0 &= 2(\sigma_{22}\sigma_{33} - \sigma^{2}_{23}) - 9h^{2}(0) + 3h'(0) - (\delta_{1} + \delta_{2}), \label{rewrited-eq:f'(0)2}\\
0 &= 6h(0)(\sigma_{22}\sigma_{33} - \sigma^{2}_{23}) - 18h'(0)h(0) + 2\sigma_{22}\delta_{2} + 2\sigma_{33}\delta_{1} + 3h''(0). \label{rewrited-eq:f''(0)2}
\end{align}
Combining \eqref{eq:h(0)}, \eqref{rewrited-eq:f'(0)2} and \eqref{rewrited-eq:f''(0)2}, we have that
\begin{equation*}
2\sigma_{33}(\delta_{1} - \delta_{2}) + 3h(0)(\delta_{1} + \delta_{2}) + 6h(0)\delta_{2} + 27h^{3}(0) - 27h'(0)h(0) + 3h''(0) = 0.
\end{equation*}
Note that $(\delta_{1} - \delta_{2}) = \frac{1}{2}(c_1 - c_2 + C(c_1 + c_2)) \neq 0$, since $C \in (-1,1)$ and $c_1 \neq c_2$. Therefore $\sigma_{33}$ is constant and hence, from \eqref{eq:h(0)} and \eqref{rewrited-eq:f'(0)2}, we have that $\sigma_{22}$ and $\sigma_{23}$ are also constant.

On the other hand, we are going to use Codazzi equation to compute $X(\sigma_{22})$, $X(\sigma_{23})$ and $X(\sigma_{33})$. As each $J_{i}$ is parallel and $A(X) = 0$ (since $\nabla C = -2A(X)$), we have $\nabla_{X}B_{j} = 0$ for all $j = 1,2,3$. In this way, since
\begin{equation*}
X(\sigma_{ij}) = X(A(B_{i},B_{j})) = \nabla A(X,B_{i},B_{j})
\end{equation*}
it follows from the Codazzi equation that
\begin{align*}
X(\sigma_{22}) &= \nabla A(B_{2},X,B_{2}) + \frac{c_{1}}{4}\{\langle X,X \rangle \langle PB_{2} + B_{2},B_{2}\rangle -\langle B_{2},X \rangle \langle PX + X,B_{2}\rangle\} \\
& \quad + \frac{c_{2}}{4}\{\langle X,X \rangle \langle PB_{2} - B_{2},B_{2}\rangle -\langle B_{2},X \rangle \langle PX - X,B_{2}\rangle\} \\
&= -C\langle AB_{2},AB_{2}\rangle + \langle PAB_{2},AB_{2}\rangle + \frac{c_{1}\lvert X \rvert^2}{2} \\
&= \frac{c_{1}(1 - C^2)}{2} + (1 - C)\sigma_{22}^2 - (1 + C)\sigma_{23}^2, 
\end{align*}
\begin{align*}
X(\sigma_{23}) &= \nabla A(B_{2},X,B_{3}) + \frac{c_{1}}{4}\{\langle X,X \rangle \langle PB_{2} + B_{2},B_{3}\rangle -\langle B_{2},X \rangle \langle PX + X,B_{3}\rangle\} \\
& \quad + \frac{c_{2}}{4}\{\langle X,X \rangle \langle PB_{2} - B_{2},B_{3}\rangle -\langle B_{2},X \rangle \langle PX - X,B_{3}\rangle\} \\
&= -C\langle AB_{2},AB_{3}\rangle + \langle PAB_{2},AB_{3}\rangle \\
&= (1 - C)\sigma_{22}\sigma_{23} - (1 + C)\sigma_{23}\sigma_{33}, 
\end{align*}
\begin{align*}
X(\sigma_{33}) &= \nabla A(B_{3},X,B_{3}) + \frac{c_{1}}{4}\{\langle X,X \rangle \langle PB_{3} + B_{3},B_{3}\rangle -\langle B_{3},X \rangle \langle PX + X,B_{3}\rangle\} \\
& \quad + \frac{c_{2}}{4}\{\langle X,X \rangle \langle PB_{3} - B_{3},B_{3}\rangle -\langle B_{3},X \rangle \langle PX - X,B_{3}\rangle\} \\
&= -C\langle AB_{3},AB_{3}\rangle + \langle PAB_{3},AB_{3}\rangle - \frac{c_{2}\lvert X \rvert^2}{2} \\
&= \frac{c_{2}(C^2 - 1)}{2} + (1 - C)\sigma_{23}^2 - (1 + C)\sigma_{33}^2.
\end{align*}
Therefore,
\begin{align}
\frac{c_{1}(1 - C^2)}{2} + (1 - C)\sigma_{22}^2 - (1 + C)\sigma_{23}^2 &= 0,\label{eq:X(sigma_{22})=0} \\
\frac{c_{2}(C^2 - 1)}{2} + (1 - C)\sigma_{23}^2 - (1 + C)\sigma_{33}^2 &= 0,\label{eq:X(sigma_{33})=0} \\
(1 - C)\sigma_{22}\sigma_{23} - (1 + C)\sigma_{23}\sigma_{33} &= 0.\label{eq:X(sigma_{23})=0}
\end{align}

Let us show that $\sigma_{23}=0$. Suppose by contradiction that $\sigma_{23} \neq 0$. From \eqref{eq:X(sigma_{23})=0}, we have
\begin{equation}\label{eq3.20}
(1 - C)^2\sigma_{22}^2 - (1 + C)^2\sigma_{33}^2 = 0.
\end{equation}
Now, multiplying \eqref{eq:X(sigma_{22})=0} by $1 - C$ and \eqref{eq:X(sigma_{33})=0} by $1 + C$, we have
\begin{align}
\frac{c_{1}(1 - C)(1 - C^2)}{2} + (1 - C)^2\sigma_{22}^2 - (1 - C^2)\sigma_{23}^2 &= 0, \label{eq3.21} \\
\frac{c_{2}(1 + C)(C^2 - 1)}{2} + (1 - C^2)\sigma_{23}^2 - (1 + C)^2\sigma_{33}^2 &= 0. \label{eq3.22}
\end{align}
Adding \eqref{eq3.21} to \eqref{eq3.22} and using \eqref{eq3.20}, we get
\begin{equation}
c_{1}(1 - C) = c_{2}(1 + C),
\end{equation}
Since $C \in (-1,1)$ and $c_{1} \neq c_{2}$, we have a contradiction. Therefore $\sigma_{23} = 0$. 

If $\sigma_{23} = 0$, the system given by equations \eqref{eq:X(sigma_{22})=0}, \eqref{eq:X(sigma_{33})=0} and \eqref{eq:X(sigma_{23})=0} is reduced to
\begin{equation}
\sigma_{22}^2 = -\frac{c_{1}(1 + C)}{2}, \quad \sigma_{33}^2 = -\frac{c_{2}(1 - C)}{2}. \label{eq:reduced-system}
\end{equation}

Observe that the only possibility of solving \eqref{eq:reduced-system} is to consider that one $c_{i}$ is negative and the other is zero. Then, without loss of generality, let us assume from now on that $c_{1} = -1$ and $c_{2} = 0$. Thus, the previous computation shows us that $\sigma_{ij}=0$, for $i \neq j$ and $\sigma_{11}=\sigma_{33}=0.$ Therefore, we conclude that $\{B_{1},B_{2},B_{3}\}$ must be a frame of principal directions of $\Sigma$, with principal curvatures
\begin{align*}
\mu_{1} = 0, \quad \mu_{2} = \pm \sqrt{\frac{1 + C}{2}}, \quad \mu_{3} = 0.
\end{align*}

In what follows, we consider the case when $\mu_{2} = \sqrt{\frac{1 + C}{2}}$. The shape operator $A$ and the tangential component of the product structure $P^T$ are given, with respect to the frame $B$, respectively by
\begin{align*}
A =
\left(
\begin{array}{ccc}
0 & 0 & 0 \\
0 & \sigma_{22} & 0 \\
0 & 0 & 0
\end{array}
\right),\quad
P^T =
\left(
\begin{array}{ccc}
-C & 0 & 0 \\
0 & 1 & 0 \\
0 & 0 & -1
\end{array}
\right).
\end{align*}

Since $P$ and $J_i$ are parallel, we have that the Levi-Civita connection $\widetilde{\nabla}$ of $\Hi^2\times\R^2$ is given by
\begin{align*}
\begin{array}{llll}
& \widetilde{\nabla}_{B_1}B_i = 0,  & \widetilde{\nabla}_{B_2}B_3 = 0, & \widetilde{\nabla}_{B_3}B_2 = 0, \\
& \widetilde{\nabla}_{B_2}B_1 = - \sqrt{\frac{1 - C}{2}}B_{2}, & \widetilde{\nabla}_{B_2}B_2 = \frac{PN + N}{\sqrt{2(1 + C)}}, & \widetilde{\nabla}_{B_3}B_1 = 0, \\
&\widetilde{\nabla}_{B_3}B_3 = 0.
\end{array}
\end{align*}
Note that $[B_1,B_3] = [B_2,B_3] = 0$. Now, let $\lambda$ a function such that
\begin{equation*}
B_1(\lambda) = -\lambda\sqrt{\frac{1 - C}{2}}, \quad B_2(\lambda) = 0 \quad \text{and} \quad B_3(\lambda) = 0.
\end{equation*}
In this way, we have 
\begin{align}\label{eq3.24}
[B_1,\lambda B_2] = \biggr(B_{1}(\lambda) + \lambda\sqrt{\frac{1 - C}{2}}\biggr)B_{2} = 0
\end{align}
and $[\lambda B_2, B_3 ] = 0$. Therefore, there is a parametrization $\Psi: \Omega \subset \R^3 \longrightarrow \Sigma$, where $\Omega$ is an open subset of $\mathbb{R}^3$ with coordinates $(t,\,u,\,v)$, such that
\begin{align*}
\Psi_{t} = B_{1}, \quad \Psi_{u} = \lambda B_{2} \quad \text{and} \quad \Psi_{v} = B_{3}.
\end{align*}
Now, we are going to construct the parametrization $\Psi$. Since $\Psi_{v} = B_{3}$, $B_{3}$ has no component in $\Hi^2$, and $\widetilde{\nabla}_{B_3}B_3 = 0$, i.e., $B_3$ is a  geodesic field of $\Hi^2 \times \R^2$, when we integrate it with respect to $v$, we have
\begin{equation*}
\Psi = \biggr(\Psi^{\Hi^2}(t,u), \beta(t,u) + B_{3}v\biggr),
\end{equation*}
where $\Psi^{\Hi^2}$ is the component of $\Psi$ in $\Hi^2$. 

Before integrating with respect to the variable $u$, we first observe that $B_2$ has no component in $\R^2$ and 
$$ 
\widetilde{\nabla}_{B_2}B_2 = \nabla^{\Hi^2}_{B_2^{\Hi^2}}B_2^{\Hi^2} = \frac{PN + N}{\sqrt{2(1 + C)}}.
$$
Therefore,
\begin{align*}
\langle \nabla^{\Hi^2}_{B_2^{\Hi^2}}B_2^{\Hi^2}, \nabla^{\Hi^2}_{B_2^{\Hi^2}}B_2^{\Hi^2} \rangle &= \frac{1}{2(1 + C)}\biggr(2\langle PN,N \rangle + \langle PN,PN \rangle + \langle N,N \rangle\biggr)\\
&=\frac{1}{2(1 + C)}(2C +2) = 1,
\end{align*}
that is, if $\varphi$ is a curve parametrized by arc length, with $\varphi' = B_2^{\Hi^2}$, then the geodesic curvature $k_g$ of $\varphi$ is $k_g = 1$, and hence $\varphi$ is a horocycle. Up to rigid motions, $\varphi$ is given by
\begin{align*}
\varphi(u) = \biggr( 1 + \frac{u^2}{2}, u, -\frac{u^2}{2}\biggr).
\end{align*}

As $\Psi_{u} = \biggr(\Psi_{u}^{\Hi^2}, \beta_{u}\biggr)= \lambda B_{2} $, it follows that $\beta$ does not depend on $u$. Thus, $\Psi_{u}^{\Hi^2} = \lambda B_{2}= \lambda(t) \left( u,\,1,\,-u,\, 0,\,0 \right)$, once $B_2(\lambda)=B_3(\lambda)=0.$ When we integrate $\Psi_{u}^{\Hi^2}$ with respect to $u$, we have 
\begin{equation*}
\Psi^{\Hi^2}(t,u) = \lambda(t)\biggr(\frac{u^2}{2},u,-\frac{u^2}{2}\biggr) + \Lambda(t),
\end{equation*}
where $\Lambda(t)$ is a smooth curve in $\Hi^2$. Hence,
\begin{equation}\label{parametrization}
\Psi(t,u,v) = \biggr(\lambda(t)\alpha(u) + \Lambda(t), \beta(t) + B_{3}v\biggr),
\end{equation}
with $\alpha(u) = \big(\frac{u^2}{2},u,-\frac{u^2}{2}\big)$. 

Finally, we integrate $B_1 = \Psi_t = \biggr(\lambda'(t)\alpha(u) + \Lambda'(t), \beta'(t)\biggr)$. Since $\widetilde{\nabla}_{B_1}B_1 = 0$, $B_1$ is also a  geodesic field of $\Hi^2 \times \R^2$. Therefore, $\beta(t) = p_{0} + V_0t$. Considering $\gamma(t) = \lambda(t)\alpha(u) + \Lambda(t)$, we have $\Psi_{t} = \biggr(\gamma'(t), V_0\biggr) = B_{1}$, with $V_0=B_{1}^{\R^2}$. It follows by the definition of $B_1$ that $\lVert B_{1}^{\R^2} \rVert = \sqrt{\frac{1 + C}{2}}$. As $\lVert \gamma' \rVert^2 + \lVert B_{1}^{\R^2} \rVert^2 = 1 $, we get $\lVert \gamma' \rVert = \lVert B_{1}^{\Hi^2} \rVert = \sqrt{\frac{1 - C}{2}} $.

Note that
\begin{align*}
\frac{D\gamma'}{dt} &= \frac{d\gamma'}{dt} - \lVert B_{1}^{\Hi^2} \rVert^2\gamma \\
&=\alpha(u)\biggr(\lambda''(t) - \lVert B_{1}^{\Hi^2} \rVert^2\lambda(t)\biggr) + \Lambda''(t) - \lVert B_{1}^{\Hi^2} \rVert^2\Lambda(t).
\end{align*}
Since $\gamma$ is a geodesic in $\Hi^2$, we have that
\begin{align*}
\lambda''(t) - \lVert B_{1}^{\Hi^2} \rVert^2\lambda(t) = 0 \quad \text{and} \quad \Lambda''(t) - \lVert B_{1}^{\Hi^2} \rVert^2\Lambda(t) = 0,
\end{align*}
and hence $\lambda(t)$ and $\Lambda(t)$ are given by
\begin{align*}
\lambda(t) &= b_{1}\cosh(r t) + b_{2}\sinh(r t), \\
\Lambda(t) &= V_{1}\cosh(r t) + V_{2}\sinh(r t),
\end{align*}
where $r = \pm \lVert B_{1}^{\Hi^2} \rVert$, $b_{i}$ are real constants and $V_{i}$ orthonormal vectors. If $\Lambda = (\Lambda_{1}, \Lambda_{2}, \Lambda_{3})$, using $\langle \gamma,\gamma\rangle = -1$, we obtain the following polynomial equation in $u$:
$$ (\lambda - (\Lambda_{1} + \Lambda_{3}))u^2 + 2\Lambda_{2}u = 0, $$
that is,
$$\lambda - (\Lambda_{1} + \Lambda_{3}) = 0 \quad \text{and} \quad \Lambda_{2}=0.$$
Therefore, if $V_{1} = (v_{11},v_{12},v_{13})$ and $V_{2} = (v_{21},v_{22},v_{23})$, we have $v_{12}=v_{22} = 0$, $b_{1} = v_{11} + v_{13}$ and $b_{1} = v_{21} + v_{23}$. Now, writing $V_{1} = (\cosh(a_{1}),0,\sinh(a_{1}))$ and $V_{2} = (\sinh(a_{1}),0,\cosh(a_{1}))$, we get $b_{1}=b_{2}=e^{a_1}$. Thus, we conclude that
\begin{align*}
\begin{split}
\lambda(t) &= e^{r t}, \\
\Lambda(t) &= \biggr(\cosh(r t),0,\sinh(r t)\biggr).
\end{split}
\end{align*}

From \eqref{eq3.24}, it follows that 
$$r + \sqrt{\dfrac{1 - C}{2}} = 0.$$
Thus, we obtain that $r = -\lVert B_{1}^{\Hi^2} \rVert$, and therefore 
\begin{align}
\begin{split}
\lambda(t) &= e^{-\lVert B_{1}^{\Hi^2} \rVert t}, \\
\Lambda(t) &= \biggr(\cosh(-\lVert B_{1}^{\Hi^2} \rVert t),0,\sinh(-\lVert B_{1}^{\Hi^2} \rVert t)\biggr).
\end{split}\label{eq3.26}
\end{align}
Writing $b = ||B^{\Hi^2}_{1}||=\sqrt{1-||B^{\R^2}_{1}||^2} = \sqrt{1-||V_0||^2}$ and $W_0 = B_{3}$, when we replace \eqref{eq3.26} in \eqref{parametrization}, we obtain the parametrization \eqref{eq:parametrization-Psi}.

For the converse, suppose that $\Sigma$ is parametrized by \eqref{eq:parametrization-Psi}. Since
\begin{align*}
\Psi_{t} &= -b \biggr( e^{-b t}(\alpha(u),\vec{0})+ \Big(\sinh(-b t), 0, \cosh(-b t), -\frac{V_0}{b} \Big)\biggr),\\
\Psi_{u} &= e^{-b t}(\alpha'(u),\vec{0}), \\
\Psi_{v} &= \Big(\vec{0},W_0 \Big),
\end{align*}
we conclude that a unit normal vector field $N$ to $\Sigma$ is given by
\begin{equation*}
N = -\lVert V_0\rVert \biggr( e^{-b t}(\alpha(u),\vec{0})+ \Big(\sinh(-b t), 0, \cosh(-b t)\Big), \frac{b}{\lVert V_0\rVert^2}V_0 \biggr).
\end{equation*}

Denoting by $\Tilde{D}$ the covariant derivative in $\mathbb{L}^3$, we obtain
\begin{align*}
\Tilde{D}_{\Psi_{t}}N &= b\lVert V_0\rVert \biggr( e^{-b t}\alpha(u) + \Big(\cosh(-b t), 0, \sinh(-b t) \Big), \vec{0} \biggr)\\
&= b\lVert V_0\rVert \Psi^{\Hi^2},\\
\Tilde{D}_{\Psi_{u}}N &= -\lVert V_0\rVert e^{-b t}(\alpha'(u),\vec{0})\\
&=-\lVert V_0\rVert \Psi_{u},\\
\Tilde{D}_{\Psi_{v}}N &= 0.
\end{align*}
It follows immediately from the derivatives above and the parametrization $\Psi$ that $$\langle \Tilde{D}_{\Psi_{u}}N,\Psi^{\Hi^2} \rangle = \langle \Tilde{D}_{\Psi_{v}}N,\Psi^{\Hi^2} \rangle = 0 \quad \text{and}\quad \langle \Tilde{D}_{\Psi_{t}}N,\Psi^{\Hi^2} \rangle = -b\lVert V_0\rVert.$$
Therefore, since $\widetilde{\nabla}_{V}W = \Tilde{D}_{V}W + \langle \Tilde{D}_{V}W, \Psi^{\Hi^2}\rangle \Psi^{\Hi^2},$ we get
\begin{align*}
\widetilde{\nabla}_{\Psi_{t}}N &= 0,\\
\widetilde{\nabla}_{\Psi_{u}}N &= -\lVert V_0\rVert \Psi_{u},\\
\widetilde{\nabla}_{\Psi_{v}}N &= 0,
\end{align*}
that is, $\Sigma$ has principal curvatures $\mu_{1} = 0, \, \mu_{2} =  \lVert V_0\rVert$ and $\mu_{3} = 0.$ Finally, since
\begin{equation*}
PN = -\lVert V_0\rVert \biggr( e^{-b t}(\alpha(u),\vec{0})+ \Big(\sinh(-b t), 0, \cosh(-b t)\Big), -\frac{b}{\lVert V_0\rVert^2}V_0 \biggr)
\end{equation*}
and $b=\sqrt{1-||V_0||^2}$, it follows that
\begin{align*}
C &= \langle PN,N \rangle \\
&= \lVert V_0\rVert^2\biggr( e^{-2b t}u^2 + e^{-b t}\Big(-\dfrac{u^2}{2}\sinh(-bt) - \dfrac{u^2}{2}\cosh(-bt)\Big) \\
& \quad \quad \quad \quad -\sinh^2(-bt) + \cosh^2(-bt) - \dfrac{b^2}{\lVert V_0\rVert^2}\biggr) \\
&= \lVert V_0\rVert^2\biggr( 1 - \dfrac{b^2}{\lVert V_0\rVert^2}\biggr)\\
&=2\lVert V_0\rVert^2 - 1,
\end{align*}
that is, $\lVert V_0\rVert = \sqrt{\dfrac{1 + C}{2}}$.
Thus, we conclude the proof of the theorem.
\end{proof}

\begin{remark} \label{rmk:geometry-psi}
Following the notation established in the proof of Theorem \ref{theo2}, let us provide a geometric description of the hypersurface given by the parametrization $\Psi$. Note that a unit normal vector to the horocycle
\begin{align*}
\varphi(u) = \biggr( 1 + \frac{u^2}{2}, u, -\frac{u^2}{2}\biggr),
\end{align*}
is given by
\begin{align*}
n(u) = \biggr( \frac{u^2}{2}, u, 1 -\frac{u^2}{2}\biggr).
\end{align*}
Fixing $u, \, v \in \mathbb{R}$, let us consider in $\Hi^2~\times~\R^2$ the following geodesic parametrized by arc length
\begin{align*}
\gamma(t) = \Big(\cosh(\omega t)\varphi(u) + \sinh(\omega t)n(u), g(v) + V_{0}t\Big),
\end{align*}
where $g(v)=p_{0} + W_{0}v$ is a geodesic in $\R^2$ with normal vector $V_{0}$. Since
\begin{align*}
\gamma'(t) = \Big(\omega\sinh(\omega t)\varphi(u) + \omega\cosh(\omega t)n(u), V_{0}\Big),
\end{align*}
it follows that
\begin{equation*}
1 = ||\gamma'(t)||^2 = \omega^2 + ||V_{0}||^2,
\end{equation*}
which implies $\omega = \pm \sqrt{1 - ||V_{0}||^2} = \pm b$. Considering $\omega = -b$, we get
\begin{align*}
\gamma(t) &= e^{-b\, t}(\alpha(u),\vec{0})+ \Big(\cosh(-b\, t), 0, \sinh(-b\, t), V_0t \Big) \\
& \quad + \Big(\vec{0}, p_0 + W_0 v \Big).
\end{align*}

Varying the parameters $(t,u,v) \in \mathbb{R}^3$, the construction above provides exactly the parametrization $\Psi$. Therefore, the hypersurface $\Psi(\mathbb{R}^3$) is a family of geodesically parallel surfaces of $\Hi^2~\times~\R^2$, given by products of horocycles in $\mathbb{H}^2$ and straight lines in $\mathbb{R}^2$.

\end{remark}

\bibliographystyle{abbrv}
\bibliography{refs}

\end{document}